\documentclass[11pt]{amsart}
\usepackage{amsmath,amsthm,amssymb,amscd,cmbright}
\usepackage{amsrefs}
\usepackage[T1]{fontenc}
\usepackage[utf8]{inputenc}
\usepackage{lmodern}
\usepackage[pdftex]{hyperref}
\usepackage{amsfonts}
\usepackage{graphicx}
\usepackage{graphics}
\usepackage[english]{babel}
\usepackage{hhline}
\usepackage[dvips]{color}
\usepackage{pb-diagram}
\DeclareGraphicsExtensions{.png,.pdf,.jpg,.gif}

\newtheorem{teor}{Theorem}[section]

\newtheorem{lemma}[teor]{Lemma}

\newtheorem{defin}[teor]{Definition}

\theoremstyle{definition}
\newtheorem{rk}[teor]{Remark}
\newtheorem{example}[teor]{Example}

\newcommand{\R}{\mbox{${\mathbb R}$}}
\newcommand{\C}{\mbox{${\mathbb C}$}}

\newcommand{\Z}{\mbox{${\mathbb Z}$}}

\newcommand{\w}{\mbox{${\omega}$}}

\newcommand{\cpi}{\mbox{${\mathbb {CP}^1}$}}

\begin{document}

\title{\bf Upper bound for the Gromov width of coadjoint orbits of type A}
\author{Alexander Caviedes Castro}
\email{alexander.caviedescastro@mail.utoronto.ca}

\date{}

\thanks{This research is partially supported by the Natural Sciences and
Engineering Research Council of Canada}

\begin{abstract}
We find an upper bound for the Gromov width of coadjoint orbits of
$U(n)$  with respect to the Kirillov-Kostant-Souriau symplectic form
by computing certain Gromov-Witten invariants. The approach
presented here is closely related to the one used by Gromov in his
celebrated non-squeezing theorem.
\end{abstract}

\maketitle

\begin{center}
\textit{This is a preliminary version. Comments are welcome.}
\end{center}

\section{Introduction}

The Darboux theorem in symplectic geometry states that around any
point of a symplectic manifold, there is a system of local
coordinates such that the symplectic manifold looks locally like
$\C^n$ with its canonical symplectic form. A natural and fundamental
problem in symplectic geometry is to know how far we can extend
symplectically these coordinates in the symplectic manifold. This is
how the concept of Gromov's width arises. The Gromov width of a
symplectic manifold $(M, \w)$ is defined as
$$
\operatorname{Gwidth}(M, \w)=\sup{\{\pi r^2: \exists \text{ a
symplectic embedding } B_{2n}(r) \hookrightarrow M \}}.
$$
Roughly speaking, the Gromov width of a symplectic manifold is a
measure of its symplectic size. Gromov's width was first introduced
by Gromov in \cite{gromov} and it has lead to the notion of
symplectic capacities \cite{hofer}.

It is interesting to know how big or small can be the Gromov width.
For example, it has been conjectured by Paul Biran that if the
cohomology class of the symplectic form of a symplectic manifold is
integral, then the Gromov width of the symplectic manifold is at
least one. On the other hand, the Gromov non-squeezing theorem gives
us insights of how restrictive is the Gromov width from above:

\textbf{Gromov's non-squeezing Theorem} \textit{If $\rho$ is a
symplectic embedding of the ball $B_{2n}(r)$ of radius $r$ into a
cylinder $B^2(\lambda)\times \R^{2n-2}$ of radius $\lambda,$ then
$r\leq \lambda.$ In particular,
$$\operatorname{Gwidth}(B^2(\lambda)\times \R^{2n-2}) = \pi
\lambda^2.$$}

Gromov's non-squeezing Theorem is frequently considered as a
classical mechanics counterpart of the Heisenberg's Uncertainty
Principle \cite{gosson}.

Gromov proved the non-squeezing theorem in \cite{gromov}, where he
established the co\-nnec\-tion between $J$-holomorphic curves and
sympletic geometry. Since then, several authors have used Gromov's
method for bounding the Gromov width of other families of symplectic
manifolds, such as G. Lu for symplectic toric manifolds in
\cite{glu2}, Yael Karshon and Susan Tolman for complex Grassmannians
manifolds in \cite{karshon} and Masrour Zoghi for regular coadjoint
orbits in \cite{masrour} (see also McDuff-Polterovich
\cite{spacking}, Biran \cite{biran}).

In this paper, we are particularly interested in finding upper
bounds for the Gromov width of general coadjoint orbits of $U(n).$
We identify the Lie algebra $\mathfrak{u}(n)$ of $U(n)$ with its
dual $\mathfrak{u}(n)^*$ via the invariant inner product defined by
the formula
$$
(X, Y)=\operatorname{Trace}{XY}.
$$
The mapping $h \mapsto ih$ is an isomorphism from the real vector
space of Hermitian matrices $\mathcal{H}:=i\mathfrak{u}(n)$ onto the
real vector space of skew-Hermitian matrices $\mathfrak{u}(n).$ This
isomorphism together with the invariant inner product allow us to
identify $\mathcal{H}$ with $\mathfrak{u}(n)^*,$ and a set of
Hermitian matrices that share the same spectrum with a coadjoint
orbit of $U(n),$ i.e., for $\lambda=(\lambda_1, \cdots,
\lambda_n)\in \R^n$ there exists a coadjoint orbit of $U(n)$ which
can be identified with $\mathcal{H}_\lambda:=\{A \in M_n(\C): A^*=A,
\operatorname{spectrum}{A}=\lambda\}.$ In this case, we can endow
$\mathcal{H}_\lambda$ with a symplectic form $\w_\lambda$ coming
from the Kostant-Kirillov-Souriau symplectic form defined on the
coadjoint orbit of $U(n)$.

The main result obtained in this paper is that if there are $i, j$
such that any difference of eigenvalues $\lambda_{i'}-\lambda_{j'}$
is an integer multiple of $\lambda_{i}-\lambda_{j},$ then
$$
\operatorname{Gwidth}(\mathcal{H}_\lambda, \w_\lambda) \leq
|\lambda_i-\lambda_j|.
$$

This result is an extension of one that Masrour Zoghi has obtained
in his Ph.D thesis \cite{masrour}, where he has considered the
problem of determining the Gromov width of \textit{regular coadjoint
orbits} of compact Lie groups. Recall that a coadjoint orbit of a
compact Lie group is regular if the stabilizer of any element of it
under the coadjoint action is a maximal torus of the compact Lie
group. When the compact Lie group is the group of unitary matrices
$U(n),$ a coadjoint orbit is regular if and only if it can be
identified with a set of the form $\mathcal{H}_\lambda$ with all the
components of $\lambda\in \R^n$ being pairwise different.  Our
results are extended to coadjoint orbits of $U(n)$ that are not
necessarily regular.

We expect to obtain a similar result for \textit{any} coadjoint
orbit of \textit{any} simple compact Lie group, but this would be
described in a later paper.

This paper is organized as follows: we first introduce the necessary
$J$-holomorphic tools that we will use throughout the text, and we
then explain how upper bounds for the Gromov width of symplectic
manifolds manifolds can be given by a non-vanishing Gromov-Witten
invariant.

Then we show how upper bounds for the Gromov width of Grassmannian
manifolds can be found by computing certain Gromov-Witten invariant.
The problem of finding the Gromov width for Grassmannians manifolds
has been already considered and solved independently by Yael Karshon
and Susan Tolman in \cite{karshon} and by Guangcun Lu in \cite{glu}.
The ideas presented in this paper are similar in nature to the ones
used by Karshon and Tolman in their paper.

Finally, we show how these considerations about Grassmannian
manifolds would be particularly useful for working out the most
general problem of determining upper bounds for the Gromov width of
partial flag manifolds. The reason of this is that in some
particular cases computations of Gromov-Witten invariants for
partial flag manifolds can be reduced to computations of
Gromov-Witten invariants for Grassmannians manifolds.

We suggest to the reader to compare our results with the results
obtained by Milena Pabiniak in \cite{milena}, where she considers
the problem of determining \textit{lower bounds} for the Gromov
width of coadjoint orbits of $U(n)$ by using equivariant techniques
of symplectic geometry. In her paper, Pabiniak proves that for
$\lambda=(\lambda_1, \cdots, \lambda_n) \in \R^n$ of the form

\begin{equation}\label{milena}
\lambda_1
> \lambda_2> \cdots > \lambda_l = \lambda_{l+1} = \cdots = \lambda_{l+s} > \lambda_{l+s+1} > \cdots > \lambda_n; s \geq
0,
\end{equation}
the Gromov width of $(\mathcal{H}_\lambda, \w_\lambda)$ is at least
the minimum $\min\{\lambda_i-\lambda_j:\lambda_i > \lambda_j\}.$
This result together with the one obtained in this paper, implies
that if $\lambda\in \R^n$ is of the form (\ref{milena}) and if there
are $i, j$ such that any difference of the form
$\lambda_{i'}-\lambda_{j'}$ is an integer multiple of
$\lambda_{i}-\lambda_{j},$ then
$$
\operatorname{Gwidth}(\mathcal{H}_\lambda,
\w_\lambda)=|\lambda_{i}-\lambda_{j}|;
$$
suggesting that the upper bound that we have found is indeed the
Gromov width of $(\mathcal{H}_\lambda, \w_\lambda).$

\textbf{Acknowledgments} I would like to thank to Yael Karshon for
letting me know about this problem and for encouraging me during the
writing process of this paper. I also would like to thank to Milena
Pabiniak for useful conversations.

\section{$J$-holomorphic curves}

Pseudoholomorphic theory has been one of the main tools used in
symplectic geometry since Gromov introduced them in \cite{gromov}
where he proved his celebrated non-squeezing theorem.  We want to
apply similar ideas for finding upper bounds for the Gromov width of
coadjoint orbits of type A, or partial flag manifolds. In this
section we give a short review of pseudoholomorphic theory and
Gromov-Witten invariants, and we show how they are related with the
Gromov width of a symplectic manifold.

\subsection{Pseudoholomorphic theory}

Let $(M^{2n}, \w)$ be a symplectic manifold. An almost complex
structure $J$ of $(M, \w)$ is a smooth operator $J:TM\to TM$ such
that $J^2=-Id.$ We say that an almost complex structure $J$ is
\textbf{compatible} with $\w$ if the formula
$$
g(v, w):=\w(v, Jw)
$$
defines a Riemannian metric. We  denote the space of $\w$-compatible
almost complex structures by $\mathcal{J}(M, \w).$

Let $(\cpi, j)$ be the Riemann sphere with its standard complex
structure $j.$ Let $J\in \mathcal{J}(M, \w).$ A map $u:\cpi \to M$
is called a \textbf{$J$-holomorphic curve of genus zero} or simply a
\textbf{$J$-holomorphic curve}  if
$$
J\circ du = du \circ j.
$$

The nonlinear Cauchy Riemman operator $\bar{\partial}$ is defined
using the formula
\begin{align*}
\bar{\partial}_J:C^{\infty}(\cpi, M) &\to \bigcup_{u\in
C^{\infty}(\mathbb{CP}^1, M)} \Omega^{0, 1}(\cpi, u^*TM) \\ u
&\mapsto \dfrac{1}{2}(d u+J\circ d u \circ j)
\end{align*}
where the codomain is considered as a bundle over $C^{\infty}(\cpi,
M),$ $\bar{\partial}_J$ is considered as a section of this bundle,
$u\in C^{\infty}(\cpi, M)$ and $u^*TM=\{(z, v):z\in \cpi, v\in
T_{u(z)}M\}.$

A curve $u:\cpi \to M$ is said to be \textbf{multiply covered} if it
is the composite of a holomorphic branched covering map $(\cpi,
j)\to (\cpi, j)$ of degree greater than one with a $J$-holomorphic
map $\cpi \to M.$ It is \textbf{simple} if it is not multiply
covered.

Given a compact symplectic manifold $(M^{2n}, \w),$ a compatible
almost complex structure $J,$ and a second homology class $A\in
H_2(M, \Z),$ we define the \textbf{moduli space of simple
$J$-holomorphic curves of degree $\mathbf{A}$} as
$$
\mathcal{M}_A^*(M, J)=\{u:\cpi \to M: J\circ du =du \circ j,
u_*[\cpi]=A, u \text{ is simple}\}.
$$

The almost complex structure $J$ is called \textbf{regular for}
$\mathbf{A}$ if for every $u\in \mathcal{M}_A^*(M, J)$ such that
$\bar{\partial}_J u=0,$ the vertical differential of the nonlinear
Cauchy-Riemann operator $\bar{\partial}_J$ at the point $u$ is
surjective onto $\Omega^{0, 1}(\cpi, u^*TM).$ If an almost complex
structure $J$ is regular for every $A \in H_2(M, \Z),$ then it will
simply be called \textbf{regular.} The set of regular
$\w$-compatible almost complex structures is residual in the set
$\mathcal{J}(M, \w)$ of compatible almost complex structures, i.e.,
it contains a countable intersection of open dense sets with respect
to the $C^{\infty}$ topology.

If $J$ is a regular almost complex structure, then the moduli space
$\mathcal{M}_A^*(M, J)$ is a smooth oriented manifold of dimension
equal to $\dim{M}+2c_1(A),$ where $c_1$ denotes the first Chern
class of the bundle $(TM, J)$ \cite{mcduff}.

\begin{example}
If $(M, \w, J)$ is a compact K\"ahler manifold and $G$ is a Lie
group such that acts transitively on $M$ by holomorphic
diffeomorphism, then the almost complex structure $J$ is regular
\cite[Proposition 7.4.3]{mcduff}.
\end{example}

A homology class $B\in H_2(M)$ is \textbf{spherical} if it is in the
image of the Hurewicz homomorphism $\pi_2(M)\to H_2(M).$ A homologiy
class $B\in H_2(M)$ is $\w$-\textbf{indecomposable} if it does not
decompose as a sum $B=B_1+\cdots +B_k$ of spherical classes such
that $\w(B_i)>0.$ Gromov's compactness theorem \cite{mcduff} implies
that when $A\in H_2(M, \Z)$ is a $\w$-indecomposable homology class
and $J$ is a regular almost complex structure,  the moduli space
$\mathcal{M}_A(M, J)/PSL(2, \C)$ of unparametrized $J$-holomorphic
curves of degree $A$ is compact.

In general, moduli spaces of pseudoholomorphic curves are not
compact but can be compactified by adding sets of stable maps
\cite{mcduff}.

The \textbf{moduli space of simple $J$-holomorphic curves of degree
$A$ with $k$-marked points} is defined by
$$
\mathcal{M}_{A, k}^*(M, J)=\mathcal{M}_A^*(M, J)\times_{PSL(2,
\mathbb{C})}(\cpi)^k
$$
where $PSL(2, \C)$ acts on the right factor by its natural action on
$\cpi$ and on the left factor by reparametrization. When $k=0,$ we
define $\mathcal{M}_{A, 0}^*(M, J)$ as being equal to
$\mathcal{M}_A(M, J)/PSL(2, \C).$ We also have an \textbf{evaluation
map}
$$
\operatorname{ev}^{k}_J:= \mathcal{M}_{0, k}^*(M, A,
J)=\mathcal{M}^*(M, A, J)\times_{PSL(2, \mathbb{C})}(\cpi)^k \to M^k
$$
defined by
$$
\operatorname{ev}^k_J[u, z_1, \cdots, z_k]=(u(z_1), \cdots, u(z_k)).
$$

A smooth homotopy of almost complex structures is a smooth family $t
\mapsto J_t, t\in [0, 1].$ For any such homotopy define
$$
\mathcal{M}^*_{A, k}(M, \{J_t\}_t)=\{(t, u):u\in \mathcal{M}^*_{A,
k}(M, J_t)\}.
$$

Given two regular $\w$-compatible almost complex structures $J_0,
J_1$ we always can find a smooth homotopy of almost complex
structures $\{J_t\}_t$ connecting them such that the space
$\mathcal{M}^*_{A, k}(M, \{J_t\}_t)$ is a smooth o\-rien\-ted
manifold of dimension $\dim{M}+2c_1(A)+2k-5$ with boundary
$\mathcal{M}^*_{A, k}(M, J_1)\sqcup \mathcal{M}^*_{A, k}(M, J_0),$
and with a smooth evaluation map
$$
\operatorname{ev}^{k}_{J_t}:\mathcal{M}_{A, k}^*(M, \{J_t\}_t)\to M
$$
such that
$$
\operatorname{ev}^k_{J_t}|_{\partial \mathcal{M}_{A, k}^*(M,
\{J_t\}_t)}=\operatorname{ev}_{J_0}^k\sqcup
\operatorname{ev}_{J_1}^k:\mathcal{M}_{A, k}^*(M,
J_1)-\mathcal{M}_{A, k}^*(M, J_0) \to M.
$$

\subsection{Gromov's width}
\begin{defin}
Given a symplectic manifold $(M^{2n}, \w),$ its Gromov's width is
defined as
$$
\operatorname{Gwidth}(M, \w)=\sup{\{\pi r^2: \exists\text{ a
symplectic embedding }  B_{2n}(r) \hookrightarrow M \}}.
$$
\end{defin}

The Darboux theorem implies that the Gromov width of a symplectic
manifold is always positive. Moreover, if the symplectic manifold is
compact, its Gromov's width is finite.

\begin{teor}\label{nsq}
Let $(M^{2n}, \w)$ be a compact symplectic manifold, and $A\in
H_2(M, \Z)\backslash \{0\}$ a second homology class. Suppose that
for a dense subset of smooth $\w$-compatible almost complex
structures, the evaluation map
$$
\operatorname{ev}_J^1:\mathcal{M}_{A, 1}^*(M, J) \to M
$$
is onto. Then for any symplectic embedding $B_{2n}(r)\hookrightarrow
M,$ we have
$$
\pi r^2 \leq \w(A),
$$
where $\w(A)$ denotes the symplectic area of $A.$ In particular,
$$
\operatorname{Gwidth}(M, \w)\leq \w(A).
$$
\end{teor}
\begin{proof}
Suppose that there is symplectic embedding
$$
\rho:B_{2n}(r)\hookrightarrow M.
$$
Fix an $\epsilon \in (0, r),$ let $\tilde{J}$ be an $\w$-compatible
complex structure on $M$ that equals $\rho_*(J_{st})$ on the open
subset $\rho(B_{2n}(r-\epsilon))\subset M.$

We claim that there exists a $\tilde{J}$-holomorphic curve
$\tilde{u} \in \mathcal{M}_B^*(M,\tilde{J})$ and $z\in \cpi$ with
$\operatorname{ev}_{\tilde{J}}^1[\tilde{u}, z]=
\tilde{u}(z)=\rho(0),$ where $0\in B_{2n}(r-\epsilon)$ is the centre
of the ball and $B \in H_2(M)$ satisfies $\w(B)\leq \w(A):$ If
$\tilde{J}$ is one of the almost complex structures for which
$\operatorname{ev}_{\tilde{J}}^1$ is onto, then we are done.
Otherwise, consider a sequence of $\w$-compatible almost complex
structures $\{J_k\}_{k=1}^{\infty}$ that $C^{\infty}$-converge to
$\tilde{J}$ and for which $\operatorname{ev}^1_{J_k}$ is onto and
choose $u_k:\cpi \to M$ such that $\rho(0)\in u_k(\cpi).$ By
Gromov's compactness, the sequence $\{(u_k, J_k)\}$ has a
subsequence $\{(u_l', J_l')\}_{l=1}^{\infty}\subset \{(u_k, J_k)\}$
Gromov converging to a stable map
$$
u^s:\cpi\sqcup\cdots \sqcup \cpi \to M
$$
whose image contains $\rho(0).$ Now, let $\tilde{u}:\cpi \to M$ be
the restriction of $u^s$ to the component of the domain of $u^s$
that contains the marked point. Moreover, let
$B=\tilde{u}_*([\cpi]),$ then it satisfies $$\w(B)\leq \w(A).$$

Since $\tilde{u}$ is $\tilde{J}$-holomorphic, its restriction to
$S:=\tilde{u}^{-1}(\rho(B_{2n}(r-\epsilon)))\subset \cpi$ gives a
proper holomorphic curve $u':S \to B_{2n}(r-\epsilon)$ that passes
through the origin. By an standard fact in minimal surface theory,
the area of this holomorphic curve is bounded from below by
$\pi(r-\epsilon)^2,$ whereas $\operatorname{area}(u')\leq
\operatorname{area}(\tilde{u})=\w(B)\leq \w(A),$ and so
$\pi(r-\epsilon)^2 \leq \w(A).$ Since this equality is true for all
$\epsilon >0,$ we conclude that
$$
\pi r^2 \leq \w(A).
$$
\end{proof}

In order to find upper bounds for the Gromov width of a symplectic
manifold $(M, \w)$, we want to prove that for generic
$\w$-compatible almost complex structures $J,$ the evaluation map
$$
\operatorname{ev}^1_J:\mathcal{M}_{A, 1}^*(M, J) \to M
$$
is onto. One way to achieve the ontoness of the evaluation map is
for example by proving that a Gromov-Witten invariant with one of
its constraints being a point is different from zero.

Gromov-Witten invariants are well defined, at least if we assume
that either the symplectic manifold $(M, \w)$ is semipositive or the
the homology class $A\in H_2(M; \Z)$ is $\w$-indecomposable, a
symplectic manifold $(M, \w)$ is \textbf{semipositive} if, for a
spherical homology class $A$ with positive symplectic area,
$c_1(A)\geq 3-n$ implies $c_1(A)\geq 0.$ In these cases, for a
regular almost complex structure $J$ of $(M, \w),$ the evaluation
map
$$
\operatorname{ev}_J^{k}:\mathcal{M}_{A, k}^*(M, J) \to M^k
$$
represents a pseudocycle, i.e., its image can be compactified by
adding a set of codimension at least two.

If $a_i \in H^*(M)$ are co\-ho\-mo\-lo\-gy classes Poincar\'e dual
to compact oriented submanifolds $X_i \subset M,$ the
\textbf{Gromov-Witten invariant} $\operatorname{GW}^J_{A, k}(a_1
\cdots a_k)$ is the number of $J$-holomorphic spheres in the class
$A$ passing through the submanifolds $X_i$ (after possibly
perturbing them) and counted with appropriate signs. More precisely,
if $\sum_{i=1}^k \deg{a_i} =\dim{\mathcal{M}_{A, k}^*(M, J)}$ and
the moduli space $\mathcal{M}_{A, k}^*(M, J)$ is endowed with a
suitable orientation (see, e.g., \cite[Section A.2]{mcduff}); the
Gromov-Witten invariant is defined as the intersection oriented
number
$$
\operatorname{GW}^J_{A, k}(a_1\cdots a_k):=\sharp
\operatorname{ev}_J^{k}\pitchfork(X_1\times \cdots\times X_k).
$$
If we do not orient the moduli space $\mathcal{M}_{A, k}^*(M, J)$,
we can still define Gromov-Witten invariants over $\Z_2.$
Gromov-Witten invariants $\operatorname{GW}^J_{A, k}$ are
well-defined, finite and independent of the regular almost complex
structure $J$ \cite[Theorem 7.1.1, Lemma 7.1.8]{mcduff}.

\begin{rk}\label{rk1}
Note that if there exist cohomology classes $a_1, \cdots, a_k$ and a
suitable regular almost complex structure $J$ such that
$\operatorname{GW}_{A, k}^J(a_1 \cdots a_k)\ne 0$ and $a_1$ is
Poincar\'e dual to the fundamental class of a point, then for a
generic choice of almost complex structure $J',$ the evaluation map
$$
\operatorname{ev}^1_{J'}:\mathcal{M}_{A, 1}^*(M, J') \to M
$$
is onto, which, by Theorem \ref{nsq}, implies that
$$
\operatorname{Gwidth}(M, \w) \leq \w(A).
$$
\end{rk}

\begin{rk}
Gromov-Witten invariants for symplectic manifolds can be defined in
wide generality by associating to the moduli spaces of
$J$-holomorphic curves \textit{virtual fundamental classes} with
rational coefficients (Li-Tian \cite{litian}, Fukaya-Ono
\cite{fukayaono}, Ruan \cite{ruan}, Siebert \cite{siebert},
Hofer-Wysocki-Zehnder \cite{hofer2}, \cite{hofer3}). We will no make
use of this de\-fi\-ni\-tion since we want to keep as simple and
self-contained as possible the presentation of this paper. However,
with this definition we would not need to assume that either the
symplectic manifold is semipositive or the homology class $A$ is
indecomposable, and the results of Theorem \ref{caviedes} can be
extended to \textit{any} coadjoint orbit of type A.
\end{rk}

\section{Coadjoint orbits of type A}\label{typeA}

The coadjoint orbits of a compact Lie group are endowed with a
symplectic form known as the KostantKirillov-Souriau form. We wish
to apply to this family of symplectic manifolds, pseudoholomorphic
tools for studying the Gromov width. We focus our attention in
coadjoint orbits of type A, or partial flag manifolds. In this
section we recall some general statements about coadjoint orbits.

Let $G$ be a compact Lie group, $\mathfrak{g}$ be its Lie algebra,
and $\mathfrak{g}^*$ be the dual of the Lie algebra $\mathfrak{g}$.
The compact Lie group $G$ acts on $\mathfrak{g}^*$ by the coadjoint
action. Let $\xi \in \mathfrak{g}^*$ and $\mathcal{O}_\xi$ be the
coadjoint orbit through $\xi.$

The coadjoint orbit $\mathcal{O}_\xi$ carries a symplectic form
defined as follows: for $\xi \in \mathfrak{g}^*$ we define a skew
bilinear form on $\mathfrak{g}$ by
$$
\w_{\xi}^{KKS}(X, Y)=\langle\xi, [X, Y] \rangle.
$$
The kernel of $\w_\xi^{KKS}$ is the Lie algebra $\mathfrak{g}_\xi$
of the stabilizer of $\xi \in \mathfrak{g}^*$ for the coadjoint
representation. In particular, $\w_\xi^{KKS}$ defines a
nondegenerate skew-symmetric bilinear form on
$\mathfrak{g}/\mathfrak{g}_\xi,$ a vector space that can be
identified with $T_\xi(\mathcal{O}_\xi)\subset \mathfrak{g}^*.$ The
bilinear form $\w_\xi^{KKS}$ induces a closed, invariant,
nondegenerate 2-form on the orbit $\mathcal{O}_\xi,$ therefore
defining a symplectic structure on $\mathcal{O}_{\xi}.$ This
symplectic form is known as the \textbf{Kostant-Kirillov-Souriau
form} of the coadjoint oribt.

Let us assume now that $G=U(n).$ Let $\mathfrak{u}(n)$ be the Lie
algebra of $U(n)$, $\mathfrak{u}(n)^*$ be its dual and
$\mathcal{H}=\{A\in M_n(\C):A^*=A\}$ be the set of Hermitian
matrices.

The group of unitary matrices $U(n)$ acts by conjugation on
$\mathcal{H}$. The Hermitian matrices $\mathcal{H}$ have real
eigenvalues and are diagonalizable in a unitary basis, so that the
orbits of this action correspond to sets of matrices in
$\mathcal{H}$ with the same spectrum. Let $\lambda=(\lambda_1,
\cdots, \lambda_n) \in\R^n$ and $\mathcal{H}_{\lambda}=\{A\in
M_n(\C): A^*=A, \operatorname{spectrum}{A}=\lambda\}$ be the
$U(n)$-orbit of the matrix $\operatorname{diagonal}(\lambda_1,
\cdots, \lambda_n)$ in $\mathcal{H}.$

We identify $U(n)$-orbits in $\mathcal{H}$ with adjoint orbits in
$\mathfrak{u}(n)$ by sending a matrix $A \in \mathcal{H}$ to the
matrix $iA\in \mathfrak{u}(n).$ The pairing in
$\mathfrak{u}(n)=i\mathcal{H} $ defined by
$$
(X, Y)=\operatorname{Trace}(XY)
$$
allows us to identify $\mathfrak{u}(n)$ with $\mathfrak{u}(n)^*,$
and adjoint orbits in $\mathfrak{u}(n)$ with coadjoint orbits in
$\mathfrak{u}(n)^*.$ So that, $U(n)$-orbits in $\mathcal{H}$ can be
identified with coadjoit orbits in $\mathfrak{u}(n)^*.$

Under these identifications, for $\lambda \in \R^n,$
$\mathcal{H}_\lambda$ can be identified with a coadjoint orbit in
$\mathfrak{u}(n)^*.$ In this case, we define a symplectic form
$\w_\lambda$ on $\mathcal{H}_\lambda$ by pulling back the
Kirillov-Kostant-Souriau form defined on the coadjoint orbit. We
also endow $\mathcal{H}_\lambda$ with a complex structure
$J_\lambda,$ coming from the presentation of $\mathcal{H}_\lambda$
as a quotient of complex Lie groups $Sl(n, \C)/P,$ where $P\subset
Sl(n, \C)$ is a parabolic subgroup of block upper triangular
matrices. The triple $(\mathcal{H}_\lambda, \w_\lambda, J_\lambda)$
is a K\"ahler manifold and the Lie group $Sl(n, \C)$ acts
holomorphically and transitively on $\mathcal{H}_\lambda$ by
conjugation.

Let $\{e_i\}_{i=1}^n$ denote the standard basis of $\R^n.$ Let
$T=U(1)^n\subset U(n)$ be the standard maximal torus of $U(n)$ and
$\mathfrak{t}\cong \R^n$ be its Lie algebra. We identify
$\mathfrak{t}^*$ with $\mathfrak{t}$ via its standard inner product
so that the standard basis $\{e_i\}_{i=1}^n$ of $\mathfrak{t}\cong
\R^n$ is identified with the standard basis of projections of
$\mathfrak{t}^*,$ which is also the standard basis (as a
$\Z$-module) of the weight lattice $\operatorname{Hom}(T,
S^1)\subset \mathfrak{t}^*.$

The restricted action of $T\subset U(n)$ on $\mathcal{H}_{\lambda}$
is Hamiltonian with momentum map
\begin{align*}
\mu: \mathcal{H}_{\lambda} &\to \mathfrak{t}^*\simeq \R^n\\(a_{ij})
&\mapsto(a_{11}, \cdots, a_{nn}).
\end{align*}
The image of the momentum map is the convex hull of the momentum
images of the fixed points of the action of $T$ on
$\mathcal{H}_{\lambda},$ i.e., the image of $\mu$ is the convex hull
of all possible permutations of the vector $(\lambda_1, \cdots,
\lambda_n)$ (see, e.g., \cite[Chapter III]{audin},
\cite{guillemin}).

The $U(n)$-orbit $\mathcal{H}_\lambda$ together with the torus $T$
action is a \textbf{GKM space}, i.e., the closure of every connected
component of the set $\{x\in
\mathcal{H}_\lambda:\dim_{\mathbb{C}}{(T\cdot x)}=1\}$ is a sphere
(see \cite{tymoczko}, \cite{gkm}). The closure of $\{x\in
\mathcal{H}_\lambda:\dim_{\mathbb{C}}{(T\cdot x)}=1\}$ is called
1-skeleton of $\mathcal{H}_\lambda.$ The \textbf{moment graph} or
\textbf{GKM graph} of $\mathcal{H}_\lambda$ is the image of its
1-skeleton under the momentum map. This graph has vertices
corresponding to the $T$-fixed points and edges corresponding to
closures of connected components of the 1-skeleton. Two vertices are
connected by an edge in the moment graph if and only if they differ
by one transposition.

For two $T$-fixed points $F, F' \in \mathcal{H}_\lambda$ such that
their images under the momentum map $\mu$ are connected by an edge
in the moment graph, we denote by $S^2_{F, F'}\subset
\mathcal{H}_\lambda$ the corresponding sphere associated to them.

We now want to compute the symplectic area of $S^2_{F, F'}\subset
\mathcal{H}_\lambda$ with respect to $\w_\lambda$ in terms of
$\lambda$. Let us suppose that $F$ and $F'$ differ by the
transposition $(i, j)\in S_n$ and the $i$-th component $F_i\in
\{\lambda_1, \cdots, \lambda_n\}$ of $F$ is greater than its $j$-th
component $F_j\in \{\lambda_1, \cdots, \lambda_n\}.$ If $T'\subset
T$ is the codimension one torus that fixes $S^2_{F, F'},$ there
exists a torus of dimension one $S\subset T$ such that $T\cong
T'\times S.$ We will use the identification $S:=\R/\Z,$ which
induces an isomorphism $\operatorname{Lie}(S)\cong \R$ leading to
$\operatorname{Lie}(S)^*\cong \R,$ mapping the lattice
$\operatorname{Hom}(S, S^1)\subset \operatorname{Lie}(S)^*$
isomorphically to $\Z\subset \R.$

The action of $S$ on $S^2_{F, F'}$ is hamiltonian with momentum map
$$\iota^*\circ \mu|_{S^2_{F, F'}}:S^2_{F, F'}\to
\operatorname{Lie}(S)^*\cong \R,$$ where $\iota:S\hookrightarrow T$
is the inclusion map. The momentum image of $S_{F, F'}^2$ under
$\iota^*\circ \mu|_{S^2_{F, F'}}$ is the segment line that joins
$\iota^*(\mu(F))$ with $\iota^*(\mu(F')).$ Note that the weight of
$T$ on $T_FS^2_{F, F'}$ is equal to $e_i-e_j,$ thus the weight of
the action of $S$ on $T_FS^2_{F, F'}$ is $\iota^*(e_i-e_j).$

Let $\gamma:[0, 1] \to S^2_{F, F'} \hookrightarrow
\mathcal{H}_\lambda$ be any smooth path from $F$ to $F'$ and $c:[0,
1]\times S \to S^2_{F, F'}$ be the map defined by $c(t, s):=s\cdot
\gamma(t).$ Then,
$$
\int_{[0, 1]\times S }c^*(\w_\lambda|_{S^2_{F, F'}})=\int_{0}^1
\gamma^*(\iota_{\xi_{S^2_{F,
F'}}\w_\lambda})=\iota^*(\mu(F))-\iota^*(\mu(F')).
$$
Note that the integral $\int_{[0, 1]\times S }c^*\w_\lambda$ is
equal to the symplectic area of $S^2_{F, F'}$ times the weight
$\iota^*(e_i-e_j).$ Since $F-F'=(F_i-F_{j})(e_{i}-e_{j}),$ and
$\iota^*(\mu(F))-\iota^*(\mu(F'))=(F_{i}-F_{j})\iota^*(e_{i}-e_{j}),$
we conclude that the symplectic area of $S^2_{F, F'}$ is equal to
$F_i-F_j$

As an example, the following figure shows the moment graph of
$\mathcal{H}_{(\lambda_1, \lambda_2,\lambda_3)}$ with three of its
edges labeled with theirs corresponding symplectic areas:

\begin{center}
\includegraphics{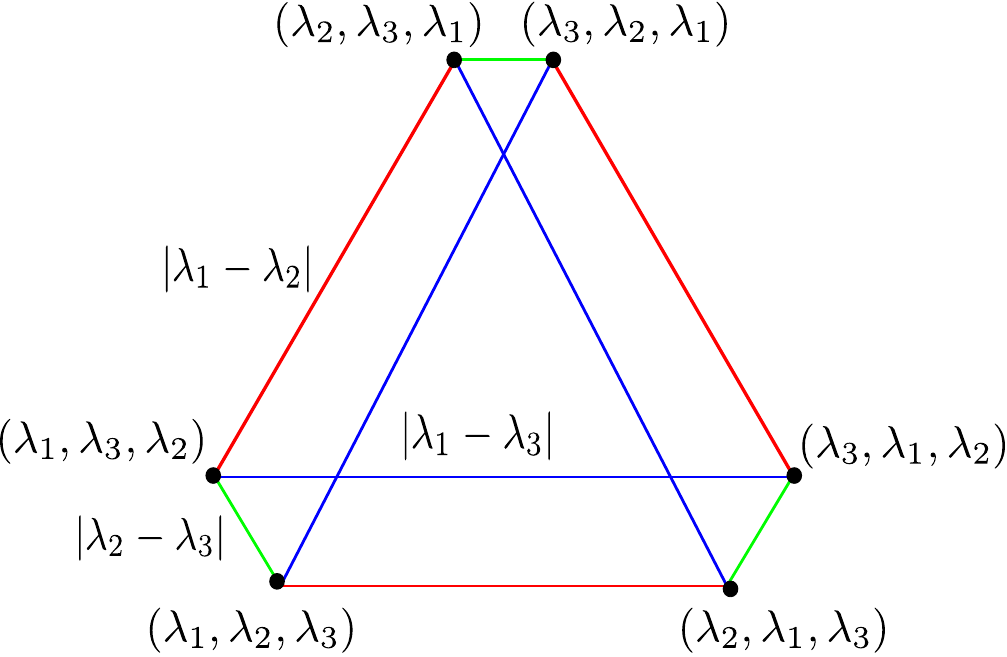}
\end{center}

Let us suppose now that $\lambda=(\lambda_1, \cdots, \lambda_n)\in
\R^n$ is of the form
$$
\lambda_1=\cdots=\lambda_{m_1},
\lambda_{m_1+1}=\cdots=\lambda_{m_1+m_2}, \cdots,
\lambda_{m_1+m_2+\cdots+m_{l-1}+1}=\cdots=\lambda_{n},
$$
where $1\leq m_1, m_2, \cdots, m_{l-1}, m_l \leq n$ are integers
such that $m_1+m_2+\cdots+m_{l-1}+m_l=n,$ and $\{\lambda_{m_1},
\lambda_{m_1+m_2}, \cdots, \lambda_{n}\}$ are all the pairwise
different components of $\lambda.$ Let $a$ be the strictly
increasing sequence of integers $0=a_0<a_1< a_2 < \cdots < a_l=n$
defined by $a_j=\sum_{i=1}^j m_i$ and let $Fl(a; n)$ be the set of
flags of type $a,$ i.e., the set of increasing filtrations of $\C^n$
by complex subspaces
$$
0=V^0\subset V^{1}\subset V^{2} \subset \cdots \subset V^{l} = \C^n
$$
such that $\dim_{\mathbb{C}}{V^i}=a_i.$ Note that there is a
naturally defined action of $Sl(n, \C)$ on $Fl(a; n).$

For a flag $V=(V^{1}, \cdots, V^{l})\in Fl(a; n),$ denote by
$P_j=P_j(V)$ the orthogonal projection onto $V_j.$ We can form the
Hermitian operator
$$
A_\lambda(V)=\sum_j \lambda_{a_j}(P_j-P_{j-1}).
$$
The correspondence $V\mapsto A_\lambda(V)$ defines a diffeomorphism
between $Fl(a; n)$ and $\mathcal{H}_\lambda.$ This diffeomorphism
defines by pullback a $U(n)$-invariant symplectic form on $Fl(a;
n).$ It also defines an integrable almost complex structure on
$Fl(a; n)$ so that $Sl(n, \C)$ acts holomorphically on $Fl(n, \C),$
and the map $A_\lambda:Fl(n, \C) \to \mathcal{H}_\lambda$ is a
$Sl(n, \C)$-invariant biholomorphism.

The (co)homology of $Fl(a, n)$ (and hence the (co)homology of
$\mathcal{H}_\lambda$) can be computed from the CW-structure of
$Fl(a; n)$ coming from its Schubert cell decomposition.

Let $S_n$ be the group of permutations of $n$ elements. Recall that
the length of a permutation is, by definition, equal to the smallest
number of adjacent transpositions whose product is the permutation.
Let $W_a \subset S_n$ be the subgroup generated by the simple
transpositions $s_i=(i, i+1)$ for $i\notin \{a_1, \cdots, a_l\}.$
Let $W^a \subset S_n$ be the set of smallest coset representatives
of $S_n/W_a.$ Let $F\in Fl(a; n)$ be the partial flag defined by
$$
F:=\C^{a_1}\subset \C^{a_2} \subset \cdots \subset \C^{a_n}=\C^n
$$
and $B$ be the standard Borel subgroup of $Sl(n, \C)$ of upper
triangular matrices.

For a permutation $w\in W^a,$ the \textbf{Schubert cell} $C_w$ is
the orbit of the induced action of $B \subset Sl(n, \C)$ on $Fl(a;
n)$ through $w\cdot F.$ The \textbf{Schubert variety} $X_w$ is by
definition the closure of the Schubert cell $C_w.$

For $w\in W^a,$ the Schubert cell $C_w$ is isomorphic to an affine
space of complex dimension equal to the length of $w.$ The Schubert
cells $\{C_w\}_{w\in W^a}$ define a CW-complex for $Fl(a; n)$ with
cells occurring only in even dimension. Thus, the fundamental
classes $[X_w]$ of $X_w, w\in W^a,$ are a free basis of $H_*(Fl(a;
n), \Z)$ as a $\Z$-module. Likewise, the Poincar\'e dual classes of
$[X_w]$, $w\in W^a,$ are a free basis of $H^*(Fl(a; n), \Z)$ as a
$\Z$-module.

The diffeomorphism $A_\lambda:Fl(a; n) \to \mathcal{H}_\lambda$ maps
the Schubert cells $C_w \in Fl(a; n), w\in W^a,$ to the $B$-orbits
of $w\cdot \lambda$ in $\mathcal{H}_\lambda.$ By abusing notation,
we will denote the $B$-orbits of $w\cdot \lambda$ in
$\mathcal{H}_\lambda$ by $C_w$ and their closures by $X_w$ and refer
to them as the Schubert cells and Schubert varieties associated to
$w\in W^a$ in $\mathcal{H}_\lambda,$ respectively.

\begin{rk}
Note that $A_\lambda$ maps the Schubert varieties
$X_{(a_j, a_j+1)} \subset Fl(a; n)$ to the spheres $S^2_{\lambda,
(a_j, a_j+1)\cdot \lambda}\subset \mathcal{H}_\lambda.$ Thus, the
homology group $H_2(\mathcal{H}_\lambda, \Z)$ is freely
ge\-ne\-ra\-ted as a $\Z$-module by the fundamental classes of
$S^2_{\lambda, (a_j, a_j+1)\cdot \lambda}, 1\leq j \leq l.$
\end{rk}

\section{Upper bounds of the Gromov width of Grassmannian ma\-ni\-folds}

Yael Karshon and Susan Tolman in \cite{karshon} found upper bounds
for the Gromov width of Grassmannian manifolds by computing a
Gromov-Witten invariant. In this section, we are going to review
this idea, which would be particularly useful for considering the
most general problem of determining upper bounds for the Gromov
width of partial flag manifolds.

We establish the convention that would be used during this section.
Let $G(k, n)$ be the Grassmannian manifold of $k$-planes in $\C^n.$
Let $\lambda \in \R^n$ be of the form
$$
\lambda_1=\cdots=\lambda_1>\lambda_2=\cdots=\cdots \lambda_2
$$
and $\mathcal{H}_\lambda=\{A\in M_n(\C): A^*=A,
\operatorname{spectrum}{A}=\lambda\}.$ As we have remarked in the
previous section, there is some integer $1\leq k \leq n$ such that
$\mathcal{H}_\lambda$ is diffeomorphic to a Grassmannian manifold
$G(k, n).$

Let $(\w_\lambda, J_\lambda)$ be the K\"ahler structure of
$\mathcal{H}_\lambda\cong G(k, n)$ defined in Section \ref{typeA}.
Let $A$ be the standard generator of the second homology group
$H_2(G(k, n), \Z).$ Let
$$
\mathcal{M}_A(J_\lambda)=\{u:\cpi \to G(k, n): u \text{ is
$J_\lambda$-holomorphic and } u_*[\cpi]=A\}
$$
be the moduli space of $J_\lambda$-holomorphic curves of degree $A$
defined on $G(k, n).$ This moduli space is usually called the space
of projective lines of the Grassmannian manifold $G(k, n).$

For a holomorphic curve $u:\cpi\to G(k, n)$ of degree $A$, we define
the kernel of $u$ as the intersection of all the subspaces $V\subset
\C^n$ that are in the image of $u.$ Similarly, the span of $u$ is
the linear span of these subspaces:
$$
\ker(u)=\bigcap_{V\in u\bigl(\mathbb{CP}^1\bigr)} V, \ \ \ \ \
\operatorname{span}(u)=\sum_{V\in u\bigl(\mathbb{CP}^1\bigr)} V.
$$
The kernel and span of $u$ are of dimension $k-1$ and $k+1$
respectively and they determine uniquely, up to parametrization, the
holomorphic curve $u,$ i.e., if there is a holomorphic curve $v
:\cpi \to G(k, n)$ of degree $A$ such that $\ker(u)=\ker(v)$ and
$\operatorname{span}(u)=\operatorname{span}(v),$ then there exists
$g:\cpi\to \cpi \in PSL(2; \C)$ such that $v=g\circ u.$ Moreover,
$u(\cpi)=\{V^k\in G(k, n): \ker(u)\subset V^k \subset
\operatorname{span}(u)\} \subset G(k, n)$ \cite{buch}. So
$\mathcal{M}_A(J_\lambda)/PSL(2, \C) \simeq Fl(k-1, k+1; n),$ where
$Fl(k-1, k+1; n)$ denotes the partial flag manifold of complex
subspaces sequences
$$
V^{k-1}\subset V^{k+1} \subset \C^n.
$$
For $V=(V^{k-1}, V^{k+1})\in Fl(k-1, k+1; n)$, we will denote by
$u_V$ the projective line
$$
\cpi \simeq u_V=\{V^k\in G(k, n): V^{k-1}\subset V^k \subset
V^{k+1}\} \subset G(k, n).
$$

Notice that $\mathcal{M}_A(J_\lambda)/PSL(2, \C)$ is compact due to
the indecomposability of $A.$ Let us consider the evaluation map
$$\operatorname{ev}^2_{J_\lambda}:\mathcal{M}_A(J_\lambda)\times_{PSL(2, \mathbb{C})} (\cpi)^2 \to G(k, n)^2.$$

We want to find a compact complex submanifold $X \subset G(k, n)$
such that for a generic point $p$ in $G(k, n)$ the evaluation map
$\operatorname{ev}^2_{J_\lambda}$ would be transverse to
$(\{p\}\times X)\subset G(k, n)^2,$
$\dim_{\mathbb{C}}(\mathcal{M}_A(J_\lambda)\times_{PSL(2,
\mathbb{C})} (\cpi)^2) + \dim_{\mathbb{C}}{X}$ would be equal to
$2\dim_{\mathbb{C}}G(k, n)$, and the number of holomorphic curves in
$\mathcal{M}_A(J_\lambda)/PSL(2, \C)$ that pass through $p$ and $X$
would be different to zero. If so, the Gromov-Witten invariant
$\operatorname{GW}_{A,2}^{J_\lambda}(\operatorname{PD}[p],
\operatorname{PD}[X])$ would be different from zero and by Theorem
\ref{nsq} and Remark \ref{rk1}, we will have that
$$
\operatorname{Gwidth}(\mathcal{H}_\lambda, \w_\lambda) \leq
\w_\lambda(A)=|\lambda_1-\lambda_2|.
$$

We claim that the Grassmannian manifold $X=\{V^k\in G(k, n):
\C\subset V^k \subset \C^{n-1}\} \subset G(k, n)$ satisfies all
these conditions.

Proving that the evaluation map $\operatorname{ev}^2_{J_\lambda}$ is
transverse to $(\{p\}\times X)\subset G(k, n)^2$ can be obtained as
a consequence of the Bertini-Kleiman Transversality Theorem:

\begin{teor}{\bf
Bertini, Kleiman \cite{mcduff}} Let $f:U \to V$ be a smooth map
between smooth manifolds and let $G$ be a Lie group that acts
transitively on $V.$ Let $Z$ be an arbitrary submanifold of $V$ and
$G^{reg}$ be the set of elements $g\in G$ for which $f$ is
transverse to $gZ.$ Then, $G^{reg}$ is a set of the second category
in $G.$
\end{teor}

We now prove that indeed the Gromov-Witten invariant
$\operatorname{GW}_{A,2}^{J_\lambda}(\operatorname{PD}[p],\operatorname{PD}[X])$
is different from zero.

\begin{lemma}\label{gw}
Let $X=\{V^k\in G(k, n): \C\subset V^k \subset \C^{n-1}\} \simeq
G(k-1, n-2)$ and $p\in G(k, n).$ Then
$$
\operatorname{GW}_{A,2}^{J_\lambda}(\operatorname{PD}[p],\operatorname{PD}[X])=1.
$$
\end{lemma}
\begin{proof}
Since the complex dimension of $X$ is equal to $(n-k-1)(k-1),$ $X$
satisfies the dimensional constraint
\begin{align*}
\dim_{\mathbb{C}}(\mathcal{M}_A(J_\lambda)\times_{PSL(2,
\mathbb{C})} (\cpi)^2) +
\dim_{\mathbb{C}}{X}&=\dim_{\mathbb{C}}Fl(k-1, k+1;
n)+2+\dim_{\mathbb{C}}X \\&= 2\dim_{\mathbb{C}}G(k, n).
\end{align*}

Assume now that $p=W^k$ is a $k$-dimensional subspace of $\C^n$ that
does not contain $\C$ and transversally intersects $\C^{n-1}.$ We
claim that $(\operatorname{ev}^2_{J_\lambda})^{-1}(\{p\}\times X)$
consists of just one element, i.e.,  there is a unique line in $G(k,
n)$ that intersects $X$ and passes through $W^k:$ let $V=(V^{k-1},
V^{k+1})\in Fl(k-1, k+1; n)$ such that the projective line $u_V$
passes through both $X$ and $p.$ So there exists $V^k \in X$ (that
is, $\C\subset V^k \subset \C^{n-1}$) and $V^{k-1} \subset V^k
\subset V^{k+1}.$ Moreover we have $V^{k-1} \subset W^k \subset
V^{k+1} $ ($W^k$ is $p$).

Note that, we have inclusions $V^{k-1}\subset \C^{n-1}$ and $V^{k-1}
\subset W^k.$ Thus $V^{k-1}\subset W^k\cap \C^{n-1}.$ But $W^k\cap
\C^{n-1}$ is a $(k-1)$-dimensional vector subspace because the
intersection is transverse. Thus $V^{k-1}=W^k\cap \C^{n-1}.$ The
intersection $V^{k-1}=W^k\cap \C^{n-1}$ does not contain $\C.$ So
there exists a unique $k$-dimensional vector space $U^k$ such that
$V^{k-1}\subset U^k$ and $\C\subset U^k \subset \C^{n-1}.$ This
vector space is $U^k=V^{k-1}\oplus \C$. Thus, $V^k=V^{k-1}\oplus
\C.$ The vector space $V^{k+1}$ contains  $W^k$ and
$V^k=V^{k-1}\oplus \C.$ Observe that $V^k$ is different from $W^k$
because $V^k$ contains $\C$ and $W^k$ does not. Therefore
$V^{k+1}=W^k+V^k.$

In conclusion $(V^{k-1}, V^{k+1})=(W^k\cap \C^{n-1}, W^k+((W^k\cap
\C^{n-1})\oplus \C)),$ which determines a unique projective line
that intersects $X$ and passes through $W^k.$

\begin{center}
\includegraphics{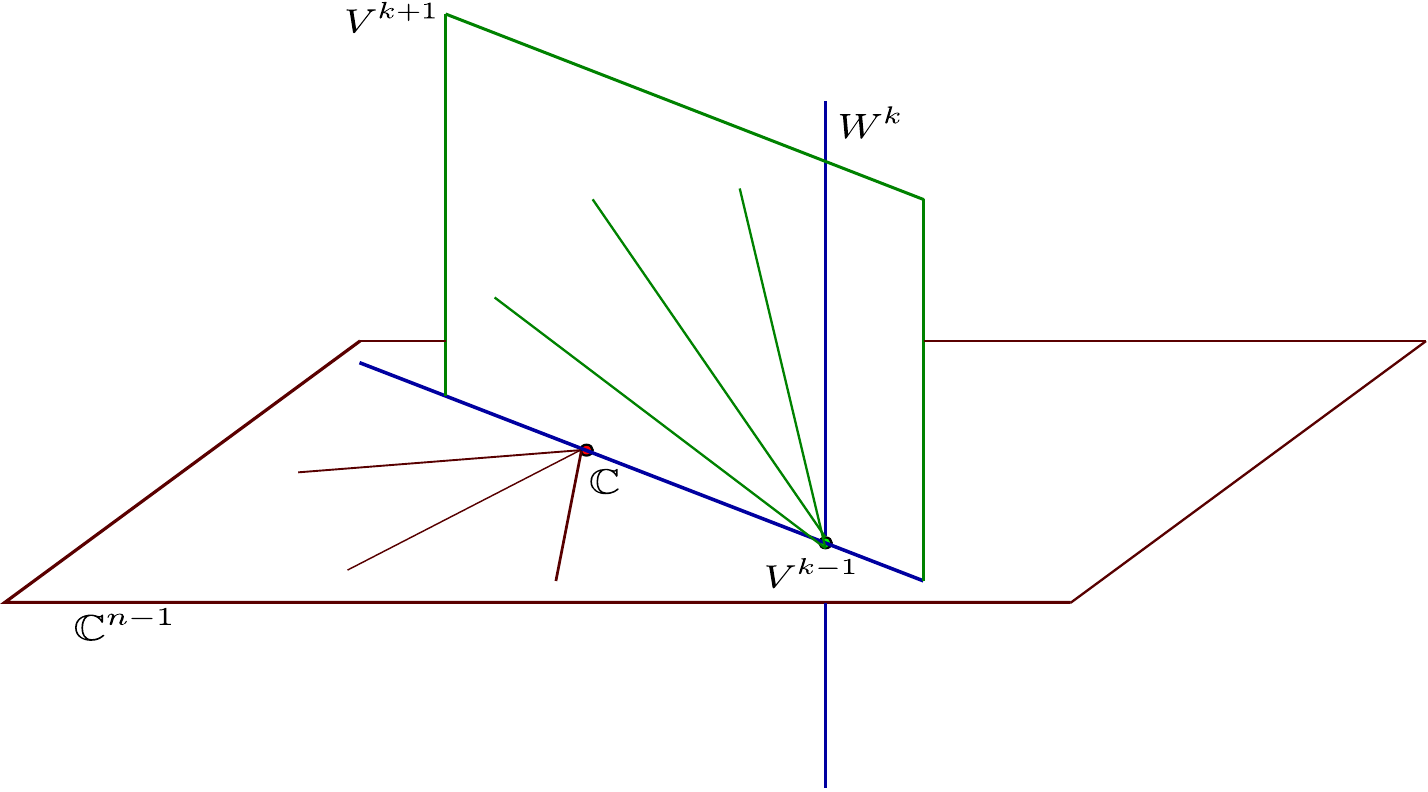}
\end{center}

Note that if $p=W^k$ is a $k$-dimensional subspace of $\C^n$ that
either contains $\C$ or is contained in $\C^{n-1},$  then
$(\operatorname{ev}^2_{J_\lambda})^{-1}(\{p\}\times X)$ consists of
an infinite number of elements.

We now prove that the evaluation map
$$
\operatorname{ev}^2_{J_\lambda}:\mathcal{M}_{A, 2}(J_\lambda) \to
G(k, n)^2
$$
is transverse to $(\{p\}\times X)\subset G(k, n)^2.$ The group
$Sl(n, \C)$ acts transitively and holomorphically on $G(k, n)$ so as
a consequence there exists $h\in Sl(n, \C)$ such that
$\operatorname{ev}^2_{J_\lambda}\pitchfork (\{h\cdot p\}\times
X)\subset G(k, n)^2$ and thus the preimage
$(\operatorname{ev}^2_{J_\lambda})^{-1}(\{h\cdot p\}\times X)$
consists of just one point (the number of elements of the preimage
$(\operatorname{ev}^2_{J_\lambda})^{-1}(\{h\cdot p\}\times X)$ is
either one or infinite, but if the evaluation map is transverse to
$\{h\cdot p\}\times X$ it has to be necessarily one), by Proposition
7.4.5 of \cite{mcduff} the Gromov-Witten invariant
$\operatorname{GW}^{J_\lambda}_{A, 2}(\operatorname{PD}[p],
\operatorname{PD}[X])$ is positive, so in conclusion
$$
\operatorname{GW}^{J_\lambda}_{A, 2}(\operatorname{PD}[p],
\operatorname{PD}[X])=\operatorname{GW}^{J_\lambda}_{A,2}(\operatorname{PD}[h\cdot
p], \operatorname{PD}[X])=1
$$
\end{proof}

We have proved that for Grassmannian manifolds there is a
non-vanishing Gromov-Witten invariant with one of its constrains
being a point. This would imply that the Gromov width of a
Grassmannian manifolds is bounded from above by the symplectic area
of any line of the Grassmannian manifold. In summary, we have the
following result:

\begin{teor}[\bf Karshon-Tolman, Guangcun Lu]\label{kt}

Let $$\mathcal{H}_\lambda=\{A\in M_n(\C):A^*=A,
\operatorname{spectrum}{A}=\lambda\}$$ where $\lambda \in \R^n$ is
of the form
$$
\lambda_1=\cdots=\lambda_1>\lambda_2=\cdots=\cdots \lambda_2,
$$
and let $\w_\lambda$ be the Kirillov-Kostant-Souriau form defined on
$\mathcal{H}_\lambda.$ Then,
$$
\operatorname{Gwidth}(\mathcal{H}_\lambda, \w_\lambda) \leq
|\lambda_1-\lambda_2|.
$$
\end{teor}
\begin{proof}
The result follows from Remark \ref{rk1} and Lemma \ref{gw}, and the
fact that the symplectic area of $A$ with respect to $\w_\lambda$ is
equal to $|\lambda_1-\lambda_2|.$
\end{proof}

\section{Upper bounds of the Gromov width of coadjoint orbits of
type A}

The problem of finding upper bounds of the Gromov width of coadjoint
orbits of type A has already been addressed by Masrour Zoghi in his
Ph.D thesis \cite{masrour} where he has considered the problem of
determining the Gromov width of regular coadjoint orbits of compact
Lie groups. We start this section by first describing Zoghi's
results, and then we show how to extend his results to coadjoint
orbits that may no be regular.

\begin{teor}\label{fibi}

Let $(M, \w)$ be a symplectic manifold and $J$ be a regular
$\w$-compatible almost complex structure on $M,$ and suppose that
$M$ admits a $J$-holomorphic $\cpi$-fibration $\pi: M \to Y$ where
$Y$ is a connected, compact K\"ahler manifold, and let $d\in H_2(M,
\Z)$ denote the homology class of the fibers of $\pi.$ Then, the
evaluation map
$$
\operatorname{ev}_J^1:\mathcal{M}_{d,1}^*(M, J)\to M
$$
is a diffeomorphism.
\end{teor}
\begin{proof}[Sketch:]
Let $u:\cpi \to M$ be a holomorphic curve of degree $d.$ Note that
$\pi_*u_*[\cpi]=\pi_*(d)=0.$ This implies that the map $\pi\circ u$
is constant because $Y$ is a connected, compact K\"ahler manifold.
As a consequence, the image of $u$ lies totally in a fiber of
$\pi:M\to Y,$ let's say $F\cong \cpi.$ The map $u: \cpi \to F \cong
\cpi$ is holomorphic of degree one; and thus $u:\cpi \to F$ is a
biholomorphism. In conclusion,  the $J$-holomorphic curves of $M$ of
degree $d,$ up to parametrization, are embedded curves in $M$ and
correspond to the fibers of $\pi:M \to Y.$

We claim (but we will no prove it) that the map $\rho:
\mathcal{M}_{d,0}^*(M, J) \to Y$ that sends one $J$-holomophirc map
$u:\cpi \to M$ to the point $\pi\circ u(\cpi)\in Y$ is indeed a
diffeomorphism.

Now, if $f: \mathcal{M}_{d, 1}^*(M, J) \to \mathcal{M}_{d, 0}^*(M,
J)$ denotes the forgetful map, the fo\-llo\-wing diagram
\[
\begin{diagram}
\node{\mathcal{M}_{d, 1}^*(M, J)} \arrow{e,t}{\operatorname{ev}_J^1}
\arrow{s,l}{f}
\node{M}\arrow{s,r}{\pi}\\
\node{\mathcal{M}_{d, 0}^*(M, J)} \arrow{e,t}{\rho} \node{Y}
\end{diagram}
\]
is commutative. It is not difficult to see that the fibers of $f$
are mapped diffeomorphically  onto the fibers of $\pi.$ This
together with the fact that $\rho: \mathcal{M}_{d, 0}^*(M, J) \to M$
is a diffeomorphism implies that $\operatorname{ev}_J^1$ is a
diffeomorphism.
\end{proof}

Let $\lambda=(\lambda_1, \cdots, \lambda_n) \in \R^n,$
$\mathcal{H}_{\lambda}=\{A\in M_n(\C): A^*=A,
\operatorname{spectrum}{A}=\lambda\}$ and $(\w_\lambda, J_\lambda)$
be the K\"ahler structure of $\mathcal{H}_\lambda$ defined in
Section \ref{typeA}.

The following theorem appears in Zoghi's Ph.D thesis \cite{masrour}
as one of its main results:
\begin{teor}[\bf Zoghi]

Let $\lambda\in \R^n$ be of the form $\lambda_1>\cdots >\lambda_n.$
Suppose that there is an integer $k$ such that any difference of
eigenvalues $\lambda_{i}-\lambda_{j}$ is an integer multiple of
$\lambda_{k+1}-\lambda_{k},$ then
$$
\operatorname{Gwidth}(\mathcal{H}_\lambda, \w_\lambda) \leq
|\lambda_k-\lambda_{k+1}|.$$
\end{teor}
\begin{proof}
The flag variety $Fl(n)$ of sequences of complex vector spaces
$$
V^1\subset V^2 \subset \cdots \subset V^n=\C^n
$$
is isomorphic to $\mathcal{H}_\lambda.$ The second homology group
$H_2(\mathcal{H}_{\lambda}, \Z)$ is freely generated by fundamental
classes of Schubert varieties $X_{(j, j+1)}$ parameterized by the
transpositions $(j, j+1)\in S_n.$ By assumption, there exists $1\leq
k < n$ such that the symplectic areas $\w_{\lambda}(X_{(i,
i+1)})=|\lambda_{i}-\lambda_{i+1}|$ are integer multiples of the
symplectic area $\w_\lambda(X_{(k, k+1)})=|\lambda_k-\lambda_{k+1}|$
for $1\leq i < n.$ This implies that $[X_{(k, k+1)}]$ is a
$\w_{\lambda}$-indecomposable homology class.

We have a naturally defined holomorphic fibration
$$\pi_k: Fl(n) \to Fl(1, \cdots, \widehat{k}, \cdots, n-1; n)$$
with fiber isomorphic to $\cpi.$ Note that the fundamental class
$[X_{(k, k+1)}]$ is the homology class of the fiber of $\pi_k.$

By Theorem \ref{fibi}, the evaluation map
$$
ev_{J_\lambda}^1:\mathcal{M}_{[X_{(k,
k+1)}],1}^*(\mathcal{H}_{\lambda}, J_\lambda)\to \mathcal{H}_\lambda
$$
is a diffeomorphism; in particular, it has degree one. Since
$[X_{(k, k+1)}]$ is a $\w_\lambda$-indecomposable homology class,
for regular $\w_\lambda$-compatible almost complex structures $J',$
the moduli spaces of $J'$-holomorphic maps $\mathcal{M}_{[X_{(k,
k+1)}],1}^*(\mathcal{H}_{\lambda}, J')$ are compact and the
evaluation maps $\operatorname{ev}_{J'}^1$ are compactly cobordant
among each other. In particular, for regular $\w_\lambda$-compatible
almost complex structures $J'$, the evaluation maps
$$
\operatorname{ev}_{J'}^1:\mathcal{M}_{[X_{(k,
k+1)}],1}^*(\mathcal{H}_{\lambda}, J')\to \mathcal{H}_\lambda
$$
have degree one and hence they are onto, which by Theorem \ref{nsq}
implies that
$$
\operatorname{Gwidth}(\mathcal{H}_\lambda, \w_\lambda) \leq
\w_\lambda[X_{(k, k+1)}]=|\lambda_k-\lambda_{k+1}|.
$$

\end{proof}

We now prove the main result of this paper, which extends Zoghi's
result to coadjoint orbits that are not necessarily regular. But
first we state the following lemma:

\begin{lemma}
Let $G=Sl(n, \C),$ $B$ be the subgroup of $G$ consisting of upper
triangular matrices and $P\subset B$ be a parabolic subgroup of
block upper triangular matrices. Let $X$ be an algebraic $G$-variety
and $\pi:X\to G/P$ be an equivariant map. If $\mathring{X}$ is the
$B$-stable open dense Schubert cell of $G/P,$ then $\pi$ is a
trivial fibration over $\mathring{X}.$
\end{lemma}
\begin{proof}
Let $x_0\in \mathring{X}$ be any point and $U\subset B$ be the
unipotent radical of $P.$ The map $s:U\to \mathring{X}$ defined by
$g\mapsto g\cdot x_0$ is an isomorphism. So that, the map
\begin{align*}
\psi:\mathring{X}\times \pi^{-1}(x_0) &\to \pi^{-1}(\mathring{X})\\
(x, y)&\mapsto s(x)\cdot y
\end{align*}
is an isomorphism with inverse given by $\psi^{-1}(m)= (\pi(m),
s(\pi(m))^{-1} \cdot m).$
\end{proof}
\begin{teor}\label{caviedes}
Let $\lambda=(\lambda_1, \cdots, \lambda_n) \in \R^n.$ Suppose that
there are $i, j$ such that any difference of eigenvalues
$\lambda_{i'}-\lambda_{j'}$ is an integer multiple of
$\lambda_{i}-\lambda_{j},$ then
$$
\operatorname{Gwidth}(\mathcal{H}_\lambda, \w_\lambda) \leq
|\lambda_i-\lambda_j|$$
\end{teor}
\begin{proof}
The idea of the proof is, as before, to prove that a certain
Gromov-Witten invariant, with one of its constraints being a point,
is different from zero.

Let us assume that $\lambda\in \R^n$ is of the form
$$
\lambda_1=\cdots=\lambda_{m_1},
\lambda_{m_1+1}=\cdots=\lambda_{m_1+m_2}, \cdots,
\lambda_{m_1+m_2+\cdots+m_{l-1}+1}=\cdots=\lambda_{n},
$$
where $1\leq m_1, m_2, \cdots, m_{l-1}, m_l \leq n$ are integers
such that $m_1+m_2+\cdots+m_{l-1}+m_l=n,$ and $\lambda_{m_1},
\lambda_{m_1+m_2}, \cdots, \lambda_{n}$ are pairwise different real
numbers. After reordering the components de $\lambda$ if necessary,
we assume that $i=m_1+1, j=m_1$ so that
$\lambda_{m_1+1}-\lambda_{m_1}$ is an integer multiple of any
difference of the form $\lambda_{i'}-\lambda_{j'}.$

We know that $\mathcal{H}_\lambda \simeq Fl(a; n),$ where $a$ is the
strictly increasing sequence of integers
$$
0=a_0<a_1<\cdots < a_l=n
$$
defined by $a_k=\sum_{r=1}^km_r,$ for $1\leq k \leq l.$

Let $a'$ be the sequence of integer numbers
$$
a_2<\cdots< a_l = n,
$$
and $Fl(a'; n)$ be the corresponding flag manifold. Let $W_a \subset
S_n$ be the subgroup generated by the simple transpositions $s_i=(i,
i+1)$ for $i\notin \{a_1, \cdots, a_l\}.$ Let $W^a \subset S_n$ be
the set of smallest coset representatives of $S_n/W_a.$ Likewise, we
define $W_{a'}$ and $W^{a'}.$ Schubert varieties of $Fl(a; n)$ and
$Fl(a'; n)$ are parametrized by $W^a$ and $W^{a'},$ respectively. To
avoid confusions, we will denote the Schubert varieties in $Fl(a;
n)$ by $X_{\bullet}$ and the Schubert varieties in $Fl(a';n)$ by
$X^{'}_{\bullet}.$ A similar thing will be done with the Schubert
cells.

For the permutations $(a_1, a_1+1)\in W^a,$ let $X_{(a_1, a_1+1)}$
be the standard Schubert variety in $Fl(a; n)$ associated to it and
let $A$ be the the fundamental class of this Schubert variety. Note
that, by assumption, $\w_\lambda(A)=|\lambda_{m_1+1}-\lambda_{m_1}|$
is a generator of the cyclic image
$\w_\lambda(H_2(\mathcal{H}_\lambda, \Z)),$ which implies that $A$
is a $\w_\lambda$-indecomposable homology class. As a consequence,
the Gromov-Witten invariant $\operatorname{GW}_{A, k}$ is well
defined.

We have a holomorphic projection
\begin{align*}
\pi:Fl(a; n) &\to Fl(a'; n)\\V^{a_1}\subset V^{a_2} \subset \cdots
\subset V^{a_l}= \C^n &\mapsto V^{a_2} \subset \cdots \subset
V^{a_l}= \C^n
\end{align*}
whose fiber is isomorphic to a Grassmanian manifold $G(a_1, a_2).$
If $G(a_1, a_2)$ is isomorphic to $\cpi,$ we are in the case of
Theorem \ref{fibi}, and we are done.

The set of minimal length representatives $W_{a}^{a'}$ of $W_{a'}$
on $W_a$ parameterizes Schubert varieties on a fiber of $\pi$. Note
that $(a_1,a_{1+1})\in W_{a}^{a'},$ so in particular $\pi_{*}(A)=0.$

Let $\tilde{w}$ be the permutation in $W_{a}^{a'}$ that represents
in a fiber a Grassmannian manifold isomorphic to $G(a_1-1, a_2-2).$
Let $w'$ be the longest element in $W^{a'}.$ The Schubert cell
$C^{'}_{w'}$ is open and dense in $Fl(a'; n).$ By the previous
Lemma, the restriction map
$$\pi|_{X_{w'\tilde{w}}}:X_{w'\tilde{w}}\to Fl(a, n)$$ is a trivial
fibration over $C_{w'}^{'}$ with fiber isomorphic to $G(a_1-1,
a_2-2).$

We now want to count the number of holomorphic curves of degree $A$
that passes through a generic point $p \in Fl(a; n)$ and
$X_{w'\tilde{w}}\subset Fl(a; n).$ Let $u:\cpi\to Fl(a; n)$ be one
of such holomorphic curves. The composition $\pi\circ u$ is
holomorphic and $(\pi\circ u)_*[\cpi]=\pi_{*}(A)=0.$ Since $Fl(a';
n)$ is a compact and connected K\"ahler manifold, the map $\pi\circ
u$ is constant, which means that the image of $u:\cpi\to Fl(a; n)$
lies entirely in the fiber $\pi^{-1}(p)\cong G(a_1, a_2)$ of
$\pi:Fl(a; n) \to Fl(a'; n).$ Moreover, $u:\cpi \to \pi^{-1}(p)\cong
G(a_1, a_2)\subset Fl(a; n)$ is a holomorphic map of degree one,
i.e., it is a projective line of the fiber $\pi^{-1}(p)\cong G(a_1,
a_2).$ If $\pi(p)\in C_{w'},$ then the fiber $\pi^{-1}(p)$
intersects $X_{w'\tilde{w}}$ in a variety isomorphic to $G(a_1-1,
a_2-2).$ Since there is just one projective line passing through a
generic point and  $G(a_1-1, a_2-2)$ in $G(a_1, a_2)$ (by Lemma
\ref{gw}), we conclude that
$$
\operatorname{GW}_{A, 2}^{J_\lambda}(\operatorname{PD}[p],
\operatorname{PD}[X_{w'\tilde{w}}])=1.
$$
Thus, by Theorem \ref{nsq} and Remark \ref{rk1},
$$
\operatorname{Gwidth}(\mathcal{H}_\lambda, \w_\lambda)\leq
\w_\lambda(A)=|\lambda_{m_1+1}-\lambda_{m_1}|.
$$
\end{proof}

\renewcommand{\refname}{Bibliography}
\bibliographystyle{plain}
\bibliography{biblo}
\nocite{*}

\end{document}